\documentclass[a4paper]{amsart}

\usepackage{amssymb}
\usepackage{amsmath,amsthm}
\usepackage{enumerate,paralist}
\usepackage{amstext}
\usepackage{psfrag}
\usepackage{dsfont}
\usepackage[dvips]{epsfig}
\usepackage[english]{babel}

\makeatletter\@addtoreset{equation}{section}\makeatother

\newcommand{\real}{{\mathds{R}}}

\newcommand{\1}{{\mathds{1}}}
\newcommand{\comp}{{\mathds{C}}}
\newcommand{\tr}{\mathop{\mathrm{tr}}}

\newcommand{\D}{\mathop{\mathrm{Dom}}}
\newcommand{\Hess}{\mathop{\mathrm{Hess}}}
\newcommand{\Div}{\mathop{\mathrm{div}}}
\newcommand{\id}{\mathop{\mathrm{Id}}}
\newcommand\liml{\lim\limits}
\newcommand\intl{\int\limits}
\newcommand\suml{\sum\limits}
\newcommand\prodl{\prod\limits}

\renewcommand{\le}{\leqslant}
\renewcommand{\ge}{\geqslant}

\newcommand\ffi{\varphi}
\newcommand\cR{\mathbb{R}}

\newcommand\pd{\partial}

\begin{document}

\newtheorem{theorem}{Theorem}[section]
\newtheorem{proposition}[theorem]{Proposition}
\newtheorem{definition}[theorem]{Definition}
\newtheorem{corollary}[theorem]{Corollary}
\newtheorem{lemma}[theorem]{Lemma}
\newtheorem{hypothesis}{Hypothesis}[section]
\theoremstyle{definition}
\newtheorem{remark}[theorem]{Remark}
\newtheorem{example}[theorem]{Example}
\newtheorem{conjecture}[theorem]{Conjecture}
\newtheorem{assumption}[theorem]{Assumtion}
\newtheorem*{ack}{Acknowledgement}

\frenchspacing

\title[Feynman formulae and phase space Feynman path integrals]{Feynman formulae  and phase space Feynman path integrals for tau-quantization of some L\'evy-Khintchine type  Hamilton functions}

\author{YANA A. BUTKO}

\address{Department of Fundamental Sciences, Bauman Moscow State Technical
University\\105005, 2nd Baumanskaya str., 5, Moscow, Russia\\
yanabutko@yandex.ru}

\author{MARTIN GROTHAUS}

\address{Fachbereich Mathematik, Technische Universit\"{a}t Kaiserslautern\\
Postfach 3049, 67653 Kaiserslautern, Germany\\
grothaus@mathematik.uni-kl.de}

\author{OLEG G. SMOLYANOV}

\address{Department of Mechanics and Mathematics,
Lomonosov Moscow State University\\119992, Vorob'evy gory, 1, Moscow, Russia\\
Smolyanov@yandex.ru}

\date{\today}

\maketitle


\begin{abstract}

This note is devoted to representation of some evolution semigroups. The semigroups are
generated by  pseudo-differential operators, which are obtained by different (parame-trized by a number $\tau$)  procedures of quantization  from a certain class of  functions (or symbols) defined on the phase space. This class contains functions which are second order polynomials with respect to the momentum variable and also some other functions. The considered semigroups are represented as limits of $n$-fold iterated integrals when $n$ tends to infinity (such representations are called Feynman formulae). Some of these representations are constructed  with the help of another pseudo-differential operators, obtained by the same procedure of quantization (such representations are called Hamiltonian Feynman formulae). Some representations are based on integral operators with elementary kernels (these ones are called Lagrangian Feynman formulae and are suitable for computations). A family of phase space Feynman pseudomeasures corresponding to different procedures of quantization  is introduced.  The considered evolution semigroups are represented also as phase space Feynman path integrals with respect to these Feynman pseudomeasures. The obtained Lagrangian Feynman formulae allow to calculate these phase space Feynman path integrals and to connect them with some functional integrals with respect to  probability measures.


\end{abstract}

\tableofcontents

\section{Introduction}

This paper is devoted to approximations of evolution semigroups $e^{-t\widehat{H}}$ generated by pseudo-differential operators $\widehat{H}$. The operators  $\widehat{H}$ are obtained from a given function $H(q,p)$  (which is called a symbol of $\widehat{H}$)  by some linear procedure (which is called a quantization). We consider a class of such procedures, parameterized by a number $\tau\in[0,1]$. This class includes $qp$-, $pq$- and Weyl quantizations. We obtain representations of the considered evolution semigroups by phase space Feynman path integrals which we define as limits of some usual integrals over finite dimensional  spaces  when the dimension of these spaces tends to infinity. Our approach is to approximate the semigroup   $e^{-t\widehat{H}}$ (for a given procedure of quantization) by a family of pseudo-differential operators  $\widehat{e^{-tH}}$ obtained by the same procedure of quantization  from the function $e^{-tH}$. Note, that if the function $H$ depends on both variables $q$   and $p$,  then  $e^{-t\widehat{H}}\neq \widehat{e^{-tH}}$. Nevertheless, under certain conditions one succeeds to prove  that
\begin{equation}\label{eq1}
e^{-t\widehat{H}}=\lim_{n\to\infty}\left[ \widehat{e^{-\frac{t}{n}H}}\right]^n.
\end{equation}
The limit in the right hand side is the limit of $n$-fold iterated integrals over the phase space when $n$ tends to infinity (such expressions are called Hamiltonian Feynman formulae). This limit
 can be interpreted as a phase space Feynman path integral with  $\exp\left\{-\intl_0^t H(q(s),p(s))ds \right\}$ in the integrand\footnote{In quantum mechanics  usually the integrands with $\exp\left\{i\intl_0^t H(q(s),p(s))ds \right\}$ are considered.}.

On a heuristic  level the same approach was   used already in Berezin's papers \cite{Ber71}, \cite{Ber80}  for investigation of Schr\"{o}dinger groups $e^{-it\widehat{H}}$. Berezin has assumed the identity
\begin{equation}\label{eq2}
e^{-it\widehat{H}}=\lim_{n\to\infty}\left[ \widehat{e^{-i\frac{t}{n}H}}\right]^n
\end{equation}
 and has interpreted the pre-limit expressions in the right hand side of the identity \eqref{eq2} as approximations to a phase space Feynman path integral. Moreover, Berezin has remarked that  Feynman path integral is ``very sensitive to the choice of approximations, and nonuniqueness appearing due to this dependence has the same character as nonuniqueness of quantization'' (see \cite{Ber80}). In other words,  Feynman path integral is different for different procedures of quantization. This difference may appear both in integrands and in the set of paths over which the integration takes place. Berezin has considered the case of Weyl quantization and his calculations have lead to a quit odd expression in the integrand of his  Feynman path integral. The question, how to distinguish the procedure of quantization on the language of Feynman path integrals, remained open.

The rigorous justification of the above mentioned approach for approximation of evolution (semi)groups was first obtained only in 2002 in the paper \cite{STT}. The main technical tool suggested in \cite{STT} was the Chernoff Theorem (see Theorem \ref{Chernoff} below, cf. \cite{Ch2}). It is a wide generalization of the classical Trotter's result used for  rigorous handling of Feynman path integrals over paths in configuration space of a system (see, e.g., \cite{Nelson}). In the paper \cite{STT} the identity \eqref{eq2} has been established for $\tau$-quantization of a class of functions $H(q,p)$ whose main ingredient is a function $h(q,p)$, which is  Fourier transform of a finite $\sigma$-additive measure. This ingredient allows to use Parseval equality to succeed the proof. A scheme to construct a phase space Feynman path integral  is also presented in \cite{STT}  (however, quit independently on the established Hamiltonian Feynman formulae \eqref{eq2}).

Later on, evolution semigroups $e^{-t\widehat{H}}$ have been treated by the same approach in papers \cite{BS}, \cite{BSchS-IDAQPRT},  \cite{BBSchS}.
In \cite{BS} the identity \eqref{eq1} has been established for the case of $qp$-quantization of  a function $H(q,p)$, which corresponds to a particle with  variable mass in a potential field. The semigroup $e^{-t\widehat{H}}$ has been considered on  the Banach space $C_\infty(\cR^d)$ of continuous, vanishing at infinity functions. The scheme of \cite{STT} was adopted (for the case of $qp$-quantization)  to interpret the obtained Hamiltonian Feynman formula \eqref{eq1} as a phase space Feynman path integral with respect to a Feynman type pseudomeasure. In \cite{BSchS-IDAQPRT} the identity \eqref{eq1} has been established for the semigroup $e^{-t\widehat{H}}$ on $C_\infty(\cR^d)$ in the case of $qp$-quantization of  a function $H(q,p)$, which is continuous  and negative definite   with respect to $p$,   continuous and bounded with respect to $q$. This class of functions $H$ contains, in particular, Hamilton functions of particles with variable mass in potential and magnetic fields and relativistic particles with variable mass. The semigroup $e^{-t\widehat{H}}$ (again on $C_\infty(\cR^d)$) generated by $\tau$-quantization of a function  $H(q,p)$, which is polynomial with respect to $p$ with variable, depending on $q$ coefficients, has been approximated in \cite{BBSchS} by a family of pseudo-differential operators with some $qp$-symbols. The obtained Hamiltonian Feynman formula has been interpreted as a phase space Feynman path integral with respect to the Feynman pseudomeasure defined in \cite{BS}.

This note continues the researches of \cite{STT}, \cite{BS}, \cite{BSchS-IDAQPRT},  \cite{BBSchS}. We consider Banach space $L_1(\cR^d)$  and  evolution semigroups $e^{-t\widehat{H}}$ generated by $\tau$-quantization of a function  $H(q,p)$, which is polynomial with respect to $p$ with variable, depending on $q$ coefficients. For all $\tau\in[0,1]$ we prove that the considered semigroups are being approximated as in \eqref{eq1}  by families of pseudo-differential operators  with $\tau$-symbols $e^{-tH}$. We construct\footnote{It is actually a modification of the scheme given in \cite{STT}.}  a family of Feynman pseudomeasures  $\Phi^\tau$, $\tau\in[0,1]$, and show that the limit in the right hand side of \eqref{eq1} for each $\tau\in[0,1]$ does coincide with a phase space Feynman path integral with respect to the corresponding pseudomeasure $\Phi^\tau$. For the case of $qp$-quantization we obtain the same result for a slightly more general class of functions $H$. We plan to obtain analogous  formulae for Schr\"{o}dinger groups $e^{-itH}$ by the method of analytic continuation in our subsequent work. The considered semigroups $e^{-t\widehat{H}}$ are represented for all $\tau\in[0,1]$ also by some limits of integral operators with (more or less) elementary kernels (such representations are called Lagrangian Feynman formulae). These representations are suitable for direct calculations. Moreover, the pre-limit expressions in the obtained Lagrangian Feynman formulae coincide with some functional integrals with respect to probability measures corresponding to stochastic processes associated to the generators $\widehat{H}$. These different representations of the same semigroups allow to calculate some phase space Feynman path integrals and to connect them with stochastic analysis.

\section{Notation and preliminaries}
\subsection{The Chernoff theorem and Feynman formulae}
Let  $(X,\|\cdot\|_X)$ be  a  Banach space,  $\mathcal{L}(X)$ be the
space of all continuous linear operators on $X$ equipped
with the strong operator topology,  $\|\cdot\|$
denote the operator norm on $\mathcal{L}(X)$ and $\id$ be the identity operator
in $X$. If $\D(L) \subset X$ is a linear subspace and $L: \D(L) \to X$
is a linear operator, then $\D(L)$ denotes the domain of $L$.  A one-parameter family $(T_t)_{t\ge0}$ of bounded linear operators
$T_t\,:X\to X$ is called a strongly continuous semigroup,  if  $T_0=\id$, $T_{s+t}=T_s\circ T_t$ for all $s,t\ge0$ and
$\lim_{t\to0}\|T_t\varphi-\varphi\|_X=0$ for all $\varphi\in X$. If $(T_t)_{t\ge0}$ is a strongly continuous semigroup on a Banach space $(X,\|\cdot\|_X)$, then the generator $L$ of $(T_t)_{t\ge0}$ is defined by
 $$
 L\varphi:=\liml_{t\to0}\frac{T_t\varphi-\varphi}{t}
 \quad
 \text{with domain}
 \,\,
  \D(L):=\bigg\{ \varphi\in X \,\bigg| \quad\liml_{t\to0}\frac{T_t\varphi-\varphi}{t}\quad\text{\rm exists as a strong limit} \bigg\}.
 $$
Consider an  evolution equation $\frac{\partial f}{\partial t}=Lf$. If $L$ is the generator of a strongly continuous semigroup  $(T_t)_{t\ge0}$ on a Banach space $(X,\|\cdot\|_X)$, then the (mild) solution of the Cauchy problem for this equation with the initial value $f(0)=f_0\in X$ is given by $f(t)=T_tf_0$ for all $f_0\in X$.
Therefore, solving the evolution equation $\frac{\partial f}{\partial t}=Lf$ means to construct a semigroup $(T_t)_{t\ge0}$ with the given generator $L$. If  the desired semigroup is not known explicitly it can be  approximated. One of the tools to approximate  semigroups is based on the Chernoff theorem \cite{Ch2}
(here we  present the version of Chernoff's theorem  given in  \cite{STT}).
\begin{theorem}[Chernoff]\label{Chernoff}
Let $X$ be a Banach space, $F:[0,\infty)\to{\mathcal{L}}(X)$ be a (strongly)
continuous mapping such that $F(0) =\id$  and $\|F(t)\|\le e^{at}$
for some  $ a\in [0, \infty)$ and all $t \ge 0$. Let
 $D$ be a linear subspace of  $\D(F'(0))$ such that the restriction of the operator
 $F'(0)$ to this subspace is closable. Let $(L, \D(L))$ be this closure. If
 $(L, \D(L))$ is the generator of a strongly continuous semigroup
 $(T_t)_{t \ge 0}$, then for any $t_0 >0$ the sequence
 $(F(t/n))^n)_{n \in {\mathds N}}$ converges to $(T_t)_{t \ge 0}$ as $n\to\infty$
 in the strong operator topology, uniformly with respect to $t\in[0,t_0],$ i.e.
 \begin{equation}\label{FF}
 T_t = \lim_{n\to\infty}\left[F(t/n)\right]^n.
 \end{equation}
\end{theorem}
Here the derivative at the origin of a function $F:[0,\varepsilon) \to
L(X)$, $\varepsilon > 0$, is a linear mapping $F'(0): \D(F'(0)) \to
X$ such that
\begin{eqnarray*}
F'(0)g := \lim_{t \to 0}\frac{F(t)g -F(0)g}{t},
\end{eqnarray*}
where $\D(F'(0))$ is the vector space of all elements $g \in X$ for
which the above limit exists.

A family of operators $(F(t))_{t \ge 0}$ suitable for the formula \eqref{FF} is called \emph{Chernoff
equivalent} to the semigroup $(T_t)_{t \ge 0}$, i.e.  this family
satisfies all the assertions of the Chernoff  theorem with respect to this semigroup.
In many cases the operators $F(t)$ are integral operators and, hence, we have a limit of iterated integrals on the right hand side of the equality \eqref{FF}. In this setting it is called  \emph{Feynman formula}.
\begin{definition}
A Feynman formula is a representation of a solution of an initial
(or initial-boundary) value problem for an evolution equation (or,
equivalently, a representation of the semigroup solving the
problem) by a limit of $n$-fold iterated integrals  as $n\to\infty$.
\end{definition}

We use this notation since it was Feynman (\cite{Feyn1}, \cite{Feyn2}) who introduced  a functional (path) integral as a limit of  iterated finite dimensional integrals. The limits in  Feynman formulae coincide with (or in some cases define) certain functional integrals with respect to probability measures or Feynman type pseudomeasures on a set  of paths of a physical system.  A  representation of a solution of an initial
(or initial-boundary) value problem for an evolution equation (or,
equivalently, a representation of the semigroup resolving the
problem) by a functional integral is usually called \emph{Feynman--Kac formula}.  Hence, the
iterated integrals in a Feynman formula for some problem give
approximations to a functional integral in the Feynman-Kac formula
representing the solution of the same problem. These approximations in many cases contain only elementary functions as integrands and, therefore,
can be used for direct calculations and simulations.

The notion of  a Feynman formula has been introduced in
\cite{STT} and the method to obtain Feynman formulae with the help of the Chernoff theorem has been developed in a series of papers
\cite{STT}--\cite{SWW5}.  Recently, this method has been
successfully applied to obtain Feynman formulae for different
classes of problems for evolution equations on different
geometric structures, see, e.g. \cite{BShS},  \cite{MZ2}--\cite{BSchS-IDAQPRT}, \cite{GS},
\cite{Obr1}, \cite{OST}, \cite{Plyash1}, \cite{Plyash2}, \cite{SakS}, \cite{SSham}, \cite{SSham10}.

We call the identity \eqref{FF}
  a \emph{Lagrangian Feynman formula}, if the $F(t)$, $t>0$,
are integral operators with elementary kernels; if the $F(t)$ are
pseudo-differential operators (the definition is given in Section \ref{PDO}), we speak of \emph{Hamiltonian Feynman formulae}.
This terminology is inspired by the fact that a Lagrangian Feynman
formula gives approximations to a functional integral over a set of
paths in the configuration space of a system (whose evolution is
described by the semigroup $(T_t)_{t \ge 0}$), while a Hamiltonian
Feynman formula corresponds to a functional integral over a set of
paths in the phase space of  some system.

\subsection{Pseudo-differential operators, their symbols and $\tau$-quantization}\label{PDO}

Let us consider a measurable function  $H\,:\,\real^d\times\real^d\to\mathds{C}$  and $\tau\in[0,1]$. We define a pseudo-differential operator ($\Psi$DO)
$\widehat{H}_\tau(\cdot,D)$ with  $\tau$-symbol $H(q,p)$  on a Banach space $(X,\|\cdot\|_X)$ of some functions on $\real^d$ by
\begin{equation}\label{def-PDO}
\widehat{H}_\tau(q,D)\varphi(q)=(2\pi)^{-d}\intl_{\real^d}\intl_{\real^d}e^{ip\cdot(q-q_1)}H(\tau q+(1-\tau)q_1,p)\varphi(q_1)\,dq_1\,dp
\end{equation}
where the domain $\D(\widehat{H}_\tau(\cdot,D))$ is the set of all $\varphi\in X$ such  that the right hand side of the  formula \eqref{def-PDO} is  well defined as an element of $(X,\|\cdot\|_X)$.  We always assume that the set of test functions $C^\infty_c(\real^d)$ belongs to the domain of the operator $\widehat{H}_\tau(\cdot,D)$.

The mapping  $H\mapsto \widehat{H}_\tau(\cdot,D)$ from a space of functions on $\real^d\times\real^d$ into the space of linear operators on $(X,\|\cdot\|_X)$  is called the $\tau$-quantization, the operator  $\widehat{H}_\tau(\cdot,D)$ is called the $\tau$-quantization of the function $H$. Note that if  the symbol $H$  is a sum of functions  depending only on one of the variables $q$ or $p$   then the $\Psi$DOs  $\widehat{H}_\tau(\cdot,D)$ coincide for all  $\tau\in[0,1]$. If $H(q,p)=q p=p q$, $q,p\in\real^1$ then $\widehat{H}_\tau(q,D)\varphi(q)=-i\tau q\frac{\partial}{\partial q}\varphi(q)-i(1-\tau)\frac{\partial}{\partial q}(q\varphi(q))$. Therefore, different $\tau$ correspond to different orderings of non-commuting operators such that  we  have  the ``qp''-quantization for $\tau=1$,    the ``pq''-quantization for $\tau=0$ and the Weyl quantization for $\tau=1/2$. A function $H(q,p)$ is usually considered  as a Hamilton function of a classical system. Then the operator $\widehat{H}_\tau(\cdot,D)$ is called the Hamiltonian of a quantum system obtained by $\tau$-quantization of the classical system with the Hamilton function $H$.

\begin{remark}
We define the Fourier transform by the formula $\mathcal{F}[\ffi](p)=(2\pi)^{-d/2}\int_{\cR^d}e^{-ip\cdot q}\ffi(q)dq$ and the inverse  Fourier transform by the formula $\mathcal{F}^{-1}[\ffi](p)=(2\pi)^{-d/2}\int_{\cR^d}e^{ip\cdot q}\ffi(q)dq$. If $\tau=1$, then formula \eqref{def-PDO} reads as
$
\widehat{H}_\tau(q,D)\varphi(q)=(2\pi)^{-d/2}\int_{\real^d}e^{ip\cdot(q)}H(q,p)\mathcal{F}[\ffi](p)\,dp.
$
In the case $H$ doesn't depend on $q$ we have also
$
\widehat{H}_\tau(q,D)\varphi(q)=\mathcal{F}^{-1}\big[H\mathcal{F}[\ffi]\big](q).
$
It is a matter of taste to chose which of the mappings $\mathcal{F}$, $\mathcal{F}^{-1}$ to consider as Fourier transform and which as its inverse.
So, the $\psi$DO with $\tau$-symbol $H(q,p)$ could be defined also like this:
\begin{equation*}
\widehat{H}_\tau(q,-D)\varphi(q)=(2\pi)^{-d}\intl_{\real^d}\intl_{\real^d}e^{-ip\cdot(q-q_1)}H(\tau q+(1-\tau)q_1,p)\varphi(q_1)\,dq_1\,dp.
\end{equation*}
It is easy to see that the operator $\widehat{H}_\tau(\cdot,-D)$ is actually a $\psi$DO  $\widehat{\lambda}_\tau(\cdot,D)$ defined by the formula \eqref{def-PDO}
with the help of the symbol $\lambda(q,p):=H(q,-p)$. If $H(q,p)$ is even with respect to $p$ then both definitions give the same operator.
\end{remark}

In the sequel we will use the following result (cf.~\cite{STT}[Lemma4]).
\begin{lemma}\label{lemma:order of PDO}
Let $\tau=1$. Let $f,g:\cR^d\to \mathbb{C}$   be  bounded continuous  functions and $\lambda:\cR^d\times\cR^d\to\mathbb{C}$ be 1-symbol of $\Psi$DO $\widehat{\lambda}_1(\cdot,D)$. Let  $H(q,p)=f(q)g(p)\lambda(q,p)$. Then
$$
\widehat{H}_1(\cdot,D)\ffi=\left(\widehat{f}\circ\widehat{\lambda}_1(\cdot,D)\circ\widehat{g}\right)\ffi
$$
for all $\ffi\in S(\cR^d)\cap \D(\widehat{H}_1(\cdot,D))\cap\D(\widehat{f}\circ\widehat{\lambda}_1(\cdot,D)\circ\widehat{g})$.
\end{lemma}

\begin{proof} Let  $\ffi\in S(\cR^d)\cap \D(\widehat{H}_1(\cdot,D))\cap\D(\widehat{f}\circ\widehat{\lambda}_1(\cdot,D)\circ\widehat{g})$. Let  $\mathcal{F}$ and $\mathcal{F}^{-1}$ stand for  Fourier transform and its inverse respectively. Then
\begin{align*}
\widehat{H}_1(q,D)\ffi(q)&=(2\pi)^{-d}\intl_{\real^d}\intl_{\real^d}e^{ip\cdot(q-q_1)}H(q,p)\varphi(q_1)\,dq_1\,dp=\\
&
=(2\pi)^{-d/2}\intl_{\real^d}e^{ip\cdot q}f(q)g(p)\lambda(q,p)\mathcal{F}[\ffi](p)dp.
\end{align*}

\begin{align*}
\left(\widehat{f}\circ\widehat{\lambda}_1(\cdot,D)\circ\widehat{g}\right)\ffi(q)&
=\left(\widehat{f}\circ\widehat{\lambda}_1(\cdot,D)\right)\mathcal{F}^{-1}[g\mathcal{F}[\ffi]](q)=\\
&
=f(q)(2\pi)^{-d/2}\intl_{\real^d}e^{ip\cdot q}\lambda(q,p)\mathcal{F}[\mathcal{F}^{-1}[g\mathcal{F}[\ffi]]](p)dp=\\
&
=(2\pi)^{-d/2}\intl_{\real^d}e^{ip\cdot q}f(q)g(p)\lambda(q,p)\mathcal{F}[\ffi](p)dp.
\end{align*}

\end{proof}

\section{ Feynman formulae for tau-quantization of some L\'evy-Khintchine type  Hamilton functions.}
    In this section we will consider $(X,\|\cdot\|_X):=(L_1(\cR^d),\|\cdot\|_1)$ and $\tau\in[0,1]$.
Let us introduce two Hamilton functions $h(q,p)$  and $H(q,p)$.
Let the function   $h(\cdot,\cdot): \real^d\times\real^d\to\comp$  be given by a formula
\begin{equation}\label{def:h}
    h(q,p)
    = c(q) +ib(q)\cdot p + p\cdot A(q)p,
\end{equation}
where for each $q\in\cR^d$ we have  $b(q)\in\real^d$, $c(q)\in\real$, $A(q)$ is a   symmetric matrix.   Let us consider also a  function  $r(\cdot): \real^d\to\comp$ given by a formula
\begin{equation}\label{def:r}
    r(p)    = \intl_{\cR^d\setminus\{ 0\}}
        \left(1-e^{iy\cdot p} + \frac{iy\cdot p}{1+|y|^2} \right)\,N(dy),
\end{equation}
where  $N(\cdot)$ is a  Radon measure on $\real^d\setminus\{0\}$ with
$\intl_{\real^d}\frac{|y|^2}{1+|y|^2}N(dy)<\infty$. Note that we consider the case when  $N$ does not depend on $q$.
Let
\begin{equation}\label{def:H}
H(q,p)=h(q,p)+r(p)\equiv c(q) +ib(q)\cdot p + p\cdot A(q)p+\intl_{\cR^d\setminus\{ 0\}}
        \left(1-e^{iy\cdot p} + \frac{iy\cdot p}{1+|y|^2} \right)\,N(dy).
\end{equation}
 If $c(q)\ge0 $ then the Hamilton function $H$ is continuous negative definite with respect to the variable $p$ and the formula \eqref{def:H}
 is just the  L\'evy-Khintchine  formula.  We don't assume in the sequel  that $c(q)\ge0 $, that's why  we call our symbol $H$ a \emph{L\'evy-Khintchine type function}.

To handle the proofs we need to assume (sometimes different) boundness and smoothness conditions on  the symbol $H$. All the assumptions, we will use in the sequel, are collected below.
\begin{assumption}\label{assum_smoothnessOfCoef} 
 (i)   There exist constants $0<a_0\le A_0<+\infty$  such that for all $p\in\cR^d$ and all $q\in\cR^d$ the following  inequalities hold
$$
a_0 |p|^2\le p\cdot A(q)p\le A_0 |p|^2.
$$

\noindent (ii) The coefficients  $A(\cdot)$, $b(\cdot)$, $c(\cdot)$  with all their derivatives up to the 4th order are continuous and bounded.

\noindent(iii) The coefficients  $A(\cdot)$, $b(\cdot)$, $c(\cdot)$ are infinite differentiable and bounded with all their derivatives.

      \noindent (iv) The symbol $H(q,p)$ is of class  $C^\infty(\cR^d)$ with respect to the variable $p$ for each $q\in \cR^d$.
\end{assumption}

\bigskip

Consider  an operator $\widehat{H}_\tau(\cdot,D)$ with the $\tau$-symbol $H(q,p)$ for $\tau\in[0,1]$ in $X$,  i.e. for any   function $\ffi\in\D(\widehat{H}_\tau(\cdot,D))\subset X$
\begin{equation}\label{def-PDO-H}
\widehat{H}_\tau(q,D)\varphi(q)=(2\pi)^{-d}\intl_{\real^d}\intl_{\real^d}e^{ip\cdot(q-q_1)}H(\tau q+(1-\tau)q_1,p)\varphi(q_1)\,dq_1\,dp.
\end{equation}

\begin{remark}\label{rem:tau connection}
Note that (due to \cite{BBSchS} and  \cite{BSchS-IDAQPRT}) the operator $\widehat{H}_\tau(\cdot,D)$ with  $\tau$-symbol  $ H(q,p)$ given by the formula \eqref{def:H} for $A(\cdot)\in C^2(\real^d)$, $b(\cdot)\in C^1(\real^d)$, $c(\cdot)\in C(\real^d)$  and  each $\tau\in[0,1]$   can be extended  to the set $C^2_b(\real^d)$  by
\begin{align}\label{formula:tau-to-1}
\widehat{H}_\tau(q,D)\varphi(q)=&
-\tr (A(q)\Hess \varphi(q))+[b(q)-2(1-\tau)\Div A(q)]\cdot\nabla\varphi(q)+\\
&+
[c(q)+(1-\tau)\Div b(q)-(1-\tau)^2\tr(\Hess A(q))]\varphi(q)+ \nonumber\\
&+\int_{y\neq 0} \left( \ffi(q+y) - \ffi(q)
    - \frac{y\cdot\nabla \ffi(q)}{1+|y|^2}\right)\,N(dy),\nonumber
\end{align}
i.e. $\widehat{H}_\tau(\cdot,D)$ is a  sum of a second order differential operator with continuous coefficients  and an integro-differential operator generating a L\'evy process.
Hence, in the case $X=C_\infty(\cR^d)$ the set $C^\infty_c(\cR^d)$ belongs to its domain (cf. \cite{Sato}[Theo.31.5]). In our  case $X=L_1(\cR^d)$ we  assume that $C^\infty_c(\cR^d)\subset\D(\widehat{H}_\tau(\cdot,D))$ and, moreover, $C^\infty_c(\cR^d)$  is a core.
\end{remark}

\begin{assumption}\label{assum_generate}
We assume that  the
coefficients $A(\cdot)$, $b(\cdot)$, $c(\cdot)$, $N$ are such that the closure $(L^\tau,\D(L^\tau))$ of a $\psi$DO $(\widehat{H}_\tau(\cdot,D),C^\infty_c(\cR^d))$ with the $\tau$-symbol $H(q,p)$ as in \eqref{def:H}  generates a strongly continuous semigroup $(T^\tau_t)_{t\ge0}$ on the space $X$ (see, e.g., Th.4.5.3, Th.4.6.25 in~\cite{NJ} for the cases when the assumption is fulfilled).
\end{assumption}

Consider a family   $(F_\tau(t))_{t\ge0}$  of $\psi$DOs with the  $\tau$-symbol
$e^{-tH(q,p)}$ in the space $X$,  i.e. for any   $\ffi\in\D(F_\tau(t))$
\begin{equation}\label{F0}
F_\tau(t)\ffi(q)=(2\pi)^{-d}\intl_{\cR^d}\intl_{\cR^d}e^{i
p\cdot (q-q_1)} e^{-{tH(\tau q+(1-\tau)q_1,p)}}\ffi(q_1)dq_1dp.
\end{equation}

\begin{lemma}\label{Lemma_tau-norm estim}
  Under Assumption \ref{assum_smoothnessOfCoef} (i), (ii) for any  $\ffi\in C^\infty_c(\cR^d)$ and any $\tau\in[0,1]$  we have    $F_\tau(t)\ffi\in X$. For all   $t\ge0$ the operators   $F_\tau(t)$ can be extended to   bounded mappings on the space $X$   and there exists a constant $k\ge0$ such that for all $t\ge0$  holds the estimate:
\begin{equation}\label{estimate:F(t)}
\|F_\tau(t)\|\le e^{tk}.
\end{equation}
\end{lemma}

\begin{proof} Using the inequalities of Assumption \ref{assum_smoothnessOfCoef} (i)  and the fact, that  the real part of each  continuous negative definite function is nonnegative (see inequalities (3.123) and (3.117) in \cite{NJ}),
we obtain the estimate
\begin{equation}\label{estimate1}
\sup\limits_{q\in \cR^d}|e^{-tH(q,p)}|\le e^{-ta_0p^2}\exp\big\{-t\min\limits_{q\in\cR^d}c(q)\big\}.
\end{equation}
Hence, the function $f_{t,q}(\cdot)=(2\pi)^{-d/2}e^{-tH(q,\cdot)}\in L_1(\cR^d)$ for each $q\in\cR^d$ and $t\ge0$. Moreover,
$f_{t,q}(0)=(2\pi)^{-d/2}e^{-tc(q)}$. Therefore,  the  inverse Fourier transform  of  $f_{t,q}(\cdot)$ has the view $e^{-tc(q)}P^q_t(\cdot)$, where for each $q\in\cR^d$  and $t\ge0$ the function $P^q_t(\cdot)\in C_\infty(\cR^d)$  is a density  of a probability measure. This follows from the Bochner Theorem and the fact that Fourier transform maps $L_1(\cR^d)$ into $C_\infty(\cR^d)$.

 Consider first the case  $\tau=0$. Then for each $\ffi\in C^\infty_c(\cR^d)$ by Fubini--Tonelly Theorem we have
 \begin{align*}
 F_0(t)\ffi(q)&=  \frac{1}{(2\pi)^d}\intl_{\cR^d}\intl_{\cR^d}e^{i
p\cdot (q-q_1)} e^{-{tH(q_1,p)}}\ffi(q_1)dq_1dp=\\
 &=\intl_{\cR^d}\ffi(q_1)e^{-tc(q_1)}P_t^{q_1}(q-q_1)dq_1.
  \end{align*}
  Again by Fubini--Tonelli Theorem
 for each $\ffi\in C^\infty_c(\cR^d)$
 \begin{align*}
 \|F_0(t)\ffi\|_1&= \bigg\|\intl_{\cR^d}\ffi(q_1)e^{-tc(q_1)}P_t^{q_1}(q-q_1)dq_1  \bigg\|_1\le\\
 &
 \le\intl_{\cR^d}\intl_{\cR^d}|\ffi(q_1)|e^{-tc(q_1)}P_t^{q_1}(q-q_1)dq_1dq=\\
 &
 =\intl_{\cR^d}|\ffi(q_1)|e^{-tc(q_1)}\bigg[\intl_{\cR^d}P_t^{q_1}(q-q_1)dq\bigg]dq_1\le\\
 &
=\exp\big\{-t\min\limits_{x\in\cR^d}c(x)\big\}\|\ffi\|_1.
  \end{align*}
Therefore, for any  $\ffi\in C^\infty_c(\cR^d)$   we have    $F_0(t)\ffi\in L_1(\cR^d)$ and $F_0(t)$ is a bounded operator from $C^\infty_c(\cR^d)$ into $L_1(\cR^d)$. Then due to the B.L.T. Theorem the operator $F_0(t)$ can be extended to a bounded operator on $L_1(\cR^d)$ with the same norm. Hence, the Lemma is true for $\tau=0$.  Let us now prove the Lemma for  the case $\tau\in(0,1]$.

\smallskip

Let us now consider  the function $H(q,p)$ given by \eqref{def:H} as a sum of functions $h(q,p)$ and $r(p)$ (see formulas \eqref{def:h}, \eqref{def:r}, \eqref{def:H}). Under Assumption \ref{assum_smoothnessOfCoef} (i), (ii) consider a family $(G^\theta_{A,b,c}(t))_{t\ge0}$ of operators on $L_1(\cR^d)$ defined  for each fixed $\theta\in(0,1]$ by the formula
\begin{equation}\label{G-theta}
\begin{aligned}
&G^\theta_{A,b,c}(t)\ffi(q)=\frac{1}{(2\pi)^d}\intl_{\cR^d}\intl_{\cR^d}e^{i
p\cdot (q-q_1)} e^{-th(\theta q,p)}\ffi(q_1)dq_1dp\equiv\\
&
\equiv\frac{e^{-tc(\theta q)}}{(4\pi t)^{d/2}(\det A(\theta q))^{1/2}}\intl_{\cR^d} \exp\bigg\{-\frac{(q-q_1-tb(\theta q))\cdot A^{-1}(\theta q)(q-q_1-tb(\theta q))}{4t}   \bigg\}\ffi(q_1)dq_1.
\end{aligned}
\end{equation}
Each $G^\theta_{A,b,c}(t)$ is a integral operator with the  kernel  $g^{\theta q}_t(q-q_1)$, where
\begin{equation}\label{gausDensity}
g^{x}_t(z)=(4\pi t)^{-d/2}(\det A(x))^{-1/2}e^{-tc(x)}\exp\bigg\{-\frac{(z-tb(x))\cdot A^{-1}(x)(z-tb(x))}{4t}\bigg\},
\end{equation}
i.e. $g_t^x$ is an inverse Fourier transform of the function $(2\pi)^{-d/2}e^{-th(x,\cdot)}$. Due to \cite{Plyash1} there is a constant $k>0$ such that  for $\theta=1$ the estimate $\|G^1_{A,b,c}(t)\|\le e^{k t}$ holds.  For each fixed $\theta\in(0,1]$ the operator $G^\theta_{A,b,c}(t)$  equals the operator $G^1_{A_\theta,b_\theta,c_\theta}(t)$ with new coefficients  $A_\theta(q)=A(\theta q)$, $b_\theta(q)=b(\theta q)$, $c_\theta(q)=c(\theta q)$ which remain as smooth and bounded as the original $A$, $b$ and $c$ are. Therefore, the estimate $\|G^\theta_{A,b,c}(t)\|\le e^{k t}$ still holds for each $\theta\in(0,1]$.
Hence,  by Fubini--Tonelli Theorem for $\ffi\in C^\infty_c(\cR^d)$ we have
\begin{align}\label{form:Lagrange-family}
&F_\tau(t)\ffi(q)= (2\pi)^{-d}\intl_{\cR^d}\intl_{\cR^d}e^{i p\cdot (q-q_1)} e^{-{tH(\tau q+(1-\tau)q_1,p)}}\ffi(q_1)dq_1dp =\nonumber\\
&
= (2\pi)^{-d}\intl_{\cR^d}\ffi(q_1)\bigg[\intl_{\cR^d}e^{i p\cdot (q-q_1)} e^{-{th(\tau q+(1-\tau)q_1,p)}}e^{-{tr(p)}}dp\bigg]dq_1 =\nonumber\\
&
=\intl_{\cR^d}\ffi(q_1)\big[g^{\tau q+(1-\tau)q_1}_t*\mu_t   \big](q-q_1)dq_1.
\end{align}
Here the function in the squared brackets   for each fixed $q,q_1\in\cR^d$ and $\tau\in(0,1]$  is an inverse Fourier transform of the  product $e^{-{th(\tau q+(1-\tau)q_1,\cdot)}}\cdot (2\pi)^{-d/2}e^{-{tr(\cdot)}}$, i.e.  a convolution  of a function $g^{\tau q+(1-\tau)q_1}_t$ given by the formula  \eqref{gausDensity} and a probability measure $\mu_t$ corresponding to the L\'{e}vy process with the symbol $r$. And this function is taken at the point $(q-q_1)$. Hence,
with $y:=q+\frac{1-\tau}{\tau}q_1$  and $x:=q_1/\tau$
\begin{align*}
&\big\|F_\tau(t)\ffi\big\|_1= \intl_{\cR^d}\bigg|\intl_{\cR^d}\ffi(q_1)\big[g^{\tau q+(1-\tau)q_1}_t*\mu_t   \big](q-q_1)dq_1\bigg|dq\le\\
&
\le\intl_{\cR^d}\intl_{\cR^d}|\ffi(q_1)|\big[g^{\tau q+(1-\tau)q_1}_t*\mu_t   \big](q-q_1)dq_1dq=\\
&
=\intl_{\cR^{d}}\bigg[\intl_{\cR^d}\intl_{\cR^d}|\ffi(q_1)|g^{\tau q+(1-\tau)q_1}_t(q-q_1-z)dq_1dq\bigg]\mu_t(dz)\le\\
&
\le\intl_{\cR^{d}}\mu_t(dz)\cdot\sup\limits_{z\in\cR^d}\bigg[\tau^d\intl_{\cR^d}\intl_{\cR^d}g^{\tau y}_t(y-z-x)|\ffi(\tau x)|dxdy\bigg]=\\
&
=\tau^d\sup\limits_{z\in\cR^d}\intl_{\cR^d} G^\tau_{A,b,c}(t)|\ffi_\tau|(y-z)dy,
\end{align*}
where  $\ffi_\tau(q):=\ffi(\tau q)$ and the operator $G^\tau_{A,b,c}$ is given by the formula \eqref{G-theta} for each $\tau\in(0,1]$.  Therefore, due to the estimate $\|G^\tau_{A,b,c}(t)\|\le e^{k t}$  for each $\ffi\in C^\infty_c(\cR^d)$ we have
\begin{align*}
\big\|F_\tau(t)\ffi\big\|_1\le\tau^d\sup\limits_{z\in\cR^d}\intl_{\cR^d} G^\tau_{A,b,c}(t)|\ffi_\tau|(y-z)dy\le \tau^d e^{kt}\|\ffi_\tau\|_1 =\tau^d e^{kt}\intl_{\cR^d}|\ffi(\tau q)|dq=e^{kt}\|\ffi\|_1.
\end{align*}
Once again by $3\varepsilon$-argument the  estimate   $\big\|F_\tau(t)\ffi\big\|_1\le e^{kt}\|\ffi\|_1$ is true for all $\ffi\in L_1(\cR^d)$.
\end{proof}

\begin{remark}
Due to results of \cite{BSchS-IDAQPRT} in the case $\tau=1$  the statement of the Lemma is also valid in the space $X=C_\infty(\cR^d)$  with $k=0$.
 Therefore, in the case $\tau=1$  by Riesz--Thorin theorem the estimate \eqref{estimate:F(t)}  holds also in all spaces $L_\mathrm{p}(\cR^d)$, $\mathrm{p}\ge1$, (with some other constants $k$).
\end{remark}

\begin{remark} \label{rem:LagrangianFF} As it follows from the representation \eqref{form:Lagrange-family}, the operators  $F_\tau(t)$ can be considered as integral operators
\begin{equation}\label{formula:LagrangianFamily}
F_\tau(t)\ffi(q)=\intl_{\cR^d}\ffi(q_1)\big[g^{\tau q+(1-\tau)q_1}_t*\mu_t   \big](q-q_1)dq_1.
\end{equation}
Hence,  this representation can be used to construct a Lagrangian Feynman formula.
\end{remark}

\begin{lemma}\label{Lemma:H:str-cont+deriv-at-0}
Let $N\equiv 0$ in the formula~\eqref{def:r}, i.e. $H(q,p)=h(q,p)$.
 Under Assumption \ref{assum_smoothnessOfCoef} (i), (ii) and  Assumption \ref{assum_generate}
for  any $\ffi\in C_c^\infty(\cR^d)$, any $\tau\in[0,1]$ and any $t_0\ge0$
 we have  $$\liml_{t\to0}\left\|\frac{F_\tau(t)\ffi-\ffi}{t}+ \widehat{H}_\tau(\cdot,D)\ffi \right\|_1=0\quad \text{ and} \quad \liml_{t\to t_0}\|F_\tau(t)\ffi-F_\tau(t_0)\ffi\|_1=0.$$
\end{lemma}
\begin{proof}
Let $\ffi\in C^\infty_c(\cR^d)\subset\D(\widehat{H}_\tau(\cdot,D))$. By Taylor's formula with $\theta\in(0,1)$ we have
\begin{align*}
&\left\|\frac{F_\tau(t)\ffi-\ffi}{t}+ \widehat{H}_\tau(\cdot,D)\ffi \right\|_1=\\
&
=t\intl_{\cR^d}\left|  (2\pi)^{-d}\intl_{\cR^d}\intl_{\cR^d}e^{i
p\cdot (q-q_1)} h^2( \tau q+(1-\tau)q_1,p)e^{-{\theta t h(\tau q+(1-\tau)q_1,p)}}\ffi(q_1)dq_1dp  \right|dq.
\end{align*}
Here $h^2(\tau q+(1-\tau)q_1,p)$ is a 4th order polynomial with respect to the variable $p$ with bounded continuously depending on $\tau q+(1-\tau)q_1$ coefficients.  Let us present  the calculations for the case $d=1$ and $b(q)\equiv0$, $c(q)\equiv0$ for simplicity. General case  can  be handled similarly.
\begin{align}\label{calculations}
&\left\|\frac{F_\tau(t)\ffi-\ffi}{t}+ \widehat{H}_\tau(\cdot,D)\ffi \right\|_1=\nonumber\\
&
=t\intl_{\cR}\left| (2\pi)^{-1}\intl_{\cR}\intl_{\cR}e^{i
p\cdot (q-q_1)} A^2(\tau q+(1-\tau)q_1)p^4e^{-{\theta t A(\tau q+(1-\tau)q_1)p^2}}\ffi(q_1)dq_1dp  \right|dq=\nonumber\\
&
=t\intl_{\cR}\left|  (2\pi)^{-1}\intl_{\cR}\intl_{\cR}\pd^4_{q_1} \left[ e^{i
p\cdot (q-q_1)}\right]\cdot\left[ A^2(\tau q+(1-\tau)q_1)e^{-{\theta t A(\tau q+(1-\tau)q_1)p^2}}\ffi(q_1)\right]dq_1dp  \right|dq=\nonumber\\
&
=t\intl_{\cR}\left|  (2\pi)^{-1}\intl_{\cR}\intl_{\cR}\left[ e^{i
p\cdot (q-q_1)}\right]\cdot \pd^4_{q_1} \left[ A^2(\tau q+(1-\tau)q_1)e^{-{\theta t A(\tau q+(1-\tau)q_1)p^2}}\ffi(q_1)\right]dq_1dp  \right|dq.
\end{align}
Consider first the case   $\tau=1$. Then
$$
\pd^4_{q_1} \left[ A^2(\tau q+(1-\tau)q_1)e^{-{\theta t A(\tau q+(1-\tau)q_1)p^2}}\ffi(q_1)\right]=
 A^2( q)e^{-{\theta t A( q)p^2}}\ffi^{(4)}(q_1)
$$
 and by the Fubini--Tonelli  theorem
\begin{align*}
&\left\|\frac{F_1(t)\ffi-\ffi}{t}+ \widehat{H}_1(\cdot,D)\ffi \right\|_1=t\intl_{\cR}\left|  (2\pi)^{-1}\intl_{\cR}\intl_{\cR}\left[ e^{i
p\cdot (q-q_1)}\right]\cdot \left[ A^2( q)e^{-{\theta t A( q)p^2}}\ffi^{(4)}(q_1)\right]dq_1dp  \right|dq=\\
&
=t \intl_{\cR}\left|  A^2( q)\intl_{\cR} (4\pi \theta t A(q))^{-1/2}e^{-\frac{(q-q_1)^2}{4\theta t A(q)}}\ffi^{(4)}(q_1)dq_1  \right|dq\le\\
&
\le t  A_0^2 \intl_{\cR} \intl_{\cR} (4\pi \theta t a_0)^{-1/2}e^{-\frac{(q-q_1)^2}{4\theta t A_0}}|\ffi^{(4)}(q_1)|dq_1  dq=t(A_0^{5/2}a_0^{-1/2})\|\ffi^{(4)}\|_1.
\end{align*}
Consider now the case when $\tau\in[0,1)$. Then
$$
\pd^4_{q_1} \left[ A^2(\tau q+(1-\tau)q_1)e^{-{\theta t A(\tau q+(1-\tau)q_1)p^2}}\ffi(q_1)\right]=
e^{-{\theta t A(\tau q+(1-\tau)q_1)p^2}}\suml_{k=0}^4 (\theta t p^2)^k\psi_k(\tau q+(1-\tau)q_1, q_1),
$$
where functions $\psi_k(x, y)$ are some linear combinations of products $A^{k}(x)(A^2)^{(m)}(x)\ffi^{(n)}(y)$ with $m,n=0,\ldots,4$. Hence $\psi_k(x,\cdot)\in C_c(\cR)$ and $\psi_k(\cdot,y)\in C_b(\cR)$ for all $x,y\in\cR$. Therefore,  with the change of variables $\sqrt{\theta t}p=\rho$, $\frac{q-q_1}{\sqrt{\theta t}}=y$ we have
\begin{align*}
&\left\|\frac{F_\tau(t)\ffi-\ffi}{t}+ \widehat{H}_\tau(\cdot,D)\ffi \right\|_1=\\
&
=t\intl_{\cR}\left| (2\pi)^{-1}\intl_{\cR}\intl_{\cR}e^{i
p\cdot (q-q_1)} e^{-{\theta t A(\tau q+(1-\tau)q_1)p^2}}\suml_{k=0}^4 (\theta t p^2)^k\psi_k(\tau q+(1-\tau)q_1, q_1)dq_1dp  \right|dq=\\
&
=t\intl_{\cR}\left| (2\pi)^{-1}\intl_{\cR}\intl_{\cR}e^{i
\rho\cdot y} e^{-{  A( q-\sqrt{\theta t}(1-\tau)y)\rho^2}}\suml_{k=0}^4 (\rho^2)^k\psi_k(q-\sqrt{\theta t}(1-\tau)y, q-\sqrt{\theta t}y)d\rho dy  \right|dq=\\
&
=t\intl_{\cR}\left| \intl_{\cR} \suml_{k=0}^4 (-1)^k \pd^{2k}_\xi \left[\frac{\exp\left\{-\frac{\xi^2}{4  A(q-\sqrt{\theta t}(1-\tau)y)}\right\}}{(4\pi   A(q-\sqrt{\theta t}(1-\tau)y))^{1/2}}\right]\bigg|_{\xi=y}
 \psi_k(q-\sqrt{\theta t}(1-\tau)y, q-\sqrt{\theta t}y) dy  \right|dq\le\\
 &
 \le t \intl_{\cR}(4\pi   a_0)^{-1/2}e^{-\frac{y^2}{4  A_0}}C_5(1+y^8)\suml_{k=0}^4 C_k \intl_{\cR}|\ffi^{(k)}(q-\sqrt{\theta t}y)|dqdy\le\\
 &
 \le t\suml_{k=0}^4 C'_k\|\ffi^{(k)}\|_1
\end{align*}
with some positive constants  $C_k$  and $C'_k$.
Analogously, for $\ffi\in C_c^\infty(\cR^d)$ by Taylor's formula with $\theta\in(0,1)$ and $t,t_0\ge0$, $t\to t_0$ we have
\begin{align*}
&\left\|F_\tau(t)\ffi-F_\tau(t_0)\ffi \right\|_1=\\
&
=|t-t_0|\intl_{\cR^d}\left|  (2\pi)^{-d}\intl_{\cR^d}\intl_{\cR^d}e^{i
p\cdot (q-q_1)} h( \tau q+(1-\tau)q_1,p)e^{-{[t_0+\theta( t-t_0)] h(\tau q+(1-\tau)q_1,p)}}\ffi(q_1)dq_1dp  \right|dq.
\end{align*}
Once again let us present  the calculations for the case $d=1$ and $b(q)\equiv0$, $c(q)\equiv0$ for simplicity. For any fixed $t_0>0$ take $t\in(t_0/2, 2t_0)$. Hence, $\alpha(t):=t_0+\theta( t-t_0)\in (t_0/2, 2t_0)$ and
\begin{align*}
&\left\|F_\tau(t)\ffi-F_\tau(t_0)\ffi \right\|_1=\\
&
=|t-t_0|\intl_{\cR}\left|  (2\pi)^{-1}\intl_{\cR}\intl_{\cR}e^{i
p\cdot (q-q_1)} A( \tau q+(1-\tau)q_1)p^2e^{-{\alpha(t) A( \tau q+(1-\tau)q_1)p^2}}\ffi(q_1)dq_1dp  \right|dq=\\
&
=|t-t_0|\intl_{\cR}\left|  \intl_{\cR}\pd^2_\xi\left[ (4\pi\alpha(t)A(\tau q+(1-\tau)q_1))^{-1/2}e^{-\frac{\xi^2}{4\alpha(t)A(\tau q+(1-\tau)q_1)}}\right]\bigg|_{\xi=q-q_1}\ffi(q_1)dq_1  \right|dq\le\\
&
\le |t-t_0|\intl_{\cR}  \intl_{\cR} (2\pi t_0a_0)^{-1/2}e^{-\frac{(q-q_1)^2}{8t_0 A_0}}C(t_0)(1+(q-q_1)^2)|\ffi(q_1)|dq_1  dq\le\\
&
\le |t-t_0|C'(t_0)\|\ffi\|_1
\end{align*}
with some  positive constants $C$ and $C'$ depending only on $t_0$.
In the case $t_0=0$ we have $F(t_0)=\id$ and we  proceed as before
\begin{align*}
&\left\|F_\tau(t)\ffi-\ffi \right\|_1=\\
&
=t\intl_{\cR}\left|  (2\pi)^{-1}\intl_{\cR}\intl_{\cR}e^{i
p\cdot (q-q_1)} A( \tau q+(1-\tau)q_1)p^2e^{-{\theta t A( \tau q+(1-\tau)q_1)p^2}}\ffi(q_1)dq_1dp  \right|dq=\\
&
=t\intl_{\cR}\left|  (2\pi)^{-1}\intl_{\cR}\intl_{\cR}\pd^2_{q_1}\left[e^{i
p\cdot (q-q_1)}\right] A( \tau q+(1-\tau)q_1)e^{-{\theta t A( \tau q+(1-\tau)q_1)p^2}}\ffi(q_1)dq_1dp  \right|dq=\\
&
=t\intl_{\cR}\left|  (2\pi)^{-1}\intl_{\cR}\intl_{\cR}e^{i
p\cdot (q-q_1)}\pd^2_{q_1}\left[A( \tau q+(1-\tau)q_1)e^{-{\theta t A( \tau q+(1-\tau)q_1)p^2}}\ffi(q_1)\right] dq_1dp  \right|dq\le\\
&
\le t\suml_{k=0}^2  C_k\|\ffi^{(k)}\|_1,
\end{align*}
where the integrals in the penultimate line can be handled as  in \eqref{calculations}.  Three epsilon argument concludes the proof  of identity $\liml_{t\to t_0}\|F_\tau(t)\ffi-F_\tau(t_0)\ffi\|_1=0$ for all $\ffi\in L_1(\cR^d)$.
\end{proof}

\begin{theorem}\label{Theo:FF-tau}
Let $X=L_1(\cR^d)$, $\tau\in[0,1]$ and $H(q,p)=h(q,p)$, where $h(q,p)$ is given by the formula \eqref{def:h}.
 Under Assumption \ref{assum_smoothnessOfCoef} (i), (ii) and Assumption \ref{assum_generate} the family $(F_\tau(t))_{t\ge0}$ given by the formula~\eqref{F0} is Chernoff equivalent to the semigroup $(T^\tau_t)_{t\ge0}$, generated by the closure $(L^\tau,\D(L^\tau))$ of a $\psi$DO $(\widehat{H}_\tau(\cdot,D),C^\infty_c(\cR^d))$ with the $\tau$-symbol $H(q,p)$. Therefore, the  Feynman formula
\begin{equation}\label{form:FF-tau}
 (T^\tau_t)\ffi=\lim_{n\to\infty}(F_\tau(t/n))^n\ffi
\end{equation}
 holds in $X=L_1(\cR^d)$ locally uniformly with respect to  $t$.  Moreover, this Feynman formula \eqref{form:FF-tau} converts into the Lagrangian one:
 \begin{align}\label{formula:LFF-tau}
 &(T^\tau_t)\ffi(q_0)=\liml_{n\to\infty}\intl_{\cR^{d}}\cdots\intl_{\cR^{d}}\ffi(q_n)\prodl_{k=1}^n g^{\tau q_{k-1}+(1-\tau)q_k}_{t/n}(q_{k-1}-q_k)dq_1\ldots dq_n,
 \end{align}
 where Gaussian type density $g^x_t(z)$  is given by the formula \eqref{gausDensity}. Additionally, under Assumption \ref{assum_smoothnessOfCoef} (iii)  we have $F_\tau(t): S(\cR^d)\to S(\cR^d)$  and the Feynman formula \eqref{form:FF-tau} with $\ffi\in S(\cR^d)$ converts also into the Hamiltonian one:
 \begin{align}\label{hff-tau}
&(T^\tau_t)\ffi(q_0)=\lim_{n\to\infty}(2\pi)^{-nd}\times\\
&
\times\intl_{\cR^{2nd}}\exp\left\{i\suml_{k=1}^n p_k\cdot (q_{k+1}-q_k)\right\} \exp\left\{-\frac{t}{n}\suml_{k=1}^n H(\tau q_{k+1}+(1-\tau)q_k,p_k)\right\}\ffi(q_1)dq_1dp_1\ldots dq_ndp_n,\nonumber
\end{align}
where $q_{n+1}:=q_0$ for all $n\in \mathbb{N}$ in the pre-limit expressions in the right hand side.
\end{theorem}
This  Theorem follows immediately from two previous Lemmas, Remark \ref{rem:LagrangianFF} and the Chernoff Theorem \ref{Chernoff}.

\bigskip
\begin{remark}
If we consider the case $H(q,p)=p\cdot Ap+c(q)$, where the matrix $A$ doesn't depend on $q$   (it is the Hamilton function of a particle with constant mass in a potential field $c$), then $\psi$DOs $\widehat{H}_\tau(\cdot,D)$ (and hence the semigroups $(T^\tau_t)_{t\ge0}$ ) do coincide for all $\tau\in[0,1]$. However, the families $(F_\tau(t))_{t\ge0}$, given by \eqref{F0}, are different since  they are $\psi$DOs whose $\tau$-symbols $e^{-t[ p\cdot Ap+c(q)]}$ nontrivially depend on both variables $q$ and $p$. Nevertheless, one can easily show, that $\|F_{\tau_1}(t)\ffi-F_{\tau_2}(t)\ffi\|_1=C(\tau_1,\tau_2)t$ with some constant $C$ depending only on $\tau_1$ and $\tau_2$ (see \cite{OSS} for deeper discussion).
\end{remark}
\begin{remark}
Let  $\tau=1$. Under Assumptions \ref{assum_smoothnessOfCoef} (i), (ii)and  Assumption \eqref{assum_generate} 
the semigroup  $(T_t^\tau)_{t\ge0}$ can be represented also by a Feynman--Kac formula  (cf. \cite{Freidlin},  \cite{Friedman}, \cite{Karatzas-Shreve}):
\begin{equation}\label{FKF-parabolic}
T^\tau_t\varphi(q_0)=\mathbb{E}_b^{q_0} \bigg[\exp\bigg(\intl_0^t
c(\xi_s)ds\bigg)\ffi(\xi_t ) \bigg],
\end{equation}
where      $\mathbb{E}_b^{q_0}$  is the expectation of a (starting at  $q_0$) diffusion process  $(\xi_t)_{t\ge0}$ with  variable diffusion matrix
${A(\cdot)}$ and drift $b(\cdot)$.  Therefore, Lagrangian Feynman formula  \eqref{formula:LFF-tau} gives (suitable for direct calculations) approximations of a functional integral in the Feynman--Kac formula  \eqref{FKF-parabolic}.  Moreover, using the representation \eqref{FKF-parabolic} for the case $b(\cdot)\equiv0 $ one can show that the expression in the right hand side of the Lagrangian Feynman formula \eqref{formula:LFF-tau} does coincide with the following functional integral
\begin{align}\label{FKF-parabolic+}
T^\tau_t\varphi(q_0)=&\mathbb{E}^{q_0} \bigg[\exp\bigg(\intl_0^t c(X_s)ds\bigg)\exp\bigg(\frac12\intl_0^t A^{-1}(X_s)b(X_s )\cdot dX_s)\bigg) \times\\
&
\phantom{qwqwqwqwqwqw}\times\exp\bigg(-\frac{1}{4}\intl_0^t A^{-1}(X_s)b(X_s)\cdot b(X_s)ds \bigg)\ffi(X_t ) \bigg],\nonumber
\end{align}
where $\mathbb{E}^{q_0}$ is the expectation of a diffusion process  $(X_t)_{t\ge0}$ with variable diffusion matrix  $A(\cdot)$ and without any  drift, a stochastic integral   $\intl_0^t A^{-1}(X_\tau)b(X_\tau )\cdot dX_\tau$ is an It\^{o} integral. Since the functional integrals in formulae  \eqref{FKF-parabolic} and \eqref{FKF-parabolic+} coincide, one obtains the analogue of the Girsanov--Cameron--Martin--Reimer--Maruyama  formula for the case of diffusion processes with variable diffusion matrices. Due to Remark \ref{rem:tau connection} the similar results are valid for all $\tau\in[0,1]$.

\end{remark}
\begin{lemma}
Consider the general case of symbol $H(q,p)=h(q,p)+r(p)$ given by the formula \eqref{def:H} and $\tau=1$. Under Assumption \ref{assum_smoothnessOfCoef} (i), (iii), (iv) and Assumption \ref{assum_generate}
for  any $\ffi\in C_c^\infty(\cR^d)$ and $t_0\ge0$
 we have  $$\liml_{t\to0}\left\|\frac{F_\tau(t)\ffi-\ffi}{t}+ \widehat{H}_\tau(\cdot,D)\ffi \right\|_1=0\quad \text{ and} \quad \liml_{t\to t_0}\|F_\tau(t)\ffi-F_\tau(t_0)\ffi\|_1=0.$$
\end{lemma}
\begin{proof}
Fix $t_0\ge0$ and let $[0,t_0+1]\ni t\to t_0$. By Taylor's formula with $\theta$ in between $t$ and $t_0$, by the Fubini--Tonelli theorem, by Lemma \ref{lemma:order of PDO} and Lemma \ref{Lemma_tau-norm estim} with a probability measure $\mu_\theta=(2\pi)^{-d/2}\mathcal{F}^{-1}[e^{-\theta r(p)}]$,   for each $\ffi\in C^\infty_c(\cR^d)$  we have
\begin{align}\label{estimate:H-strCont}
&\|F_1(t)\ffi-F_1(t_0)\ffi\|_1=\\
&
=\left\|  \frac{t-t_0}{(2\pi)^d}\intl_{\cR^d}\intl_{\cR^d}e^{ip\cdot (q-q_1)} H( q,p)e^{-{\theta H(q,p)}}\ffi(q_1)dq_1dp  \right\|_1=\nonumber\\
&=|t-t_0|\left\|  \left(\widehat{He^{-{\theta H}}}\right)_1(\cdot,D)\ffi  \right\|_1\le\nonumber\\
&
\le|t-t_0|\left[\left\|\left(\widehat{he^{-\theta H}}\right)_1(\cdot,D)\ffi  \right\|_1+\left\|\left(\widehat{re^{-\theta H}}\right)_1(\cdot,D)\ffi  \right\|_1\right]=\nonumber\\
&
=|t-t_0|\left[\left\|\left(\widehat{he^{-\theta h}}\right)_1(\cdot,D)\circ \left(\widehat{e^{-\theta r}}\right)\ffi  \right\|_1+\left\|\left(\widehat{e^{-\theta H}}\right)_1(\cdot,D)\circ \widehat{r}\,\ffi  \right\|_1\right]\le\nonumber\\
&
\le |t-t_0|\left[\left\|\left(\widehat{he^{-\theta h}}\right)_1(\cdot,D)(\mu_\theta*\ffi)  \right\|_1+\|F_1(\theta)\|\,\| \widehat{r}\,\ffi  \|_1\right]\le\nonumber\\
&
\le|t-t_0|\left[ \suml_{k=0}^2 C_k(t_0)\|(\mu_\theta*\ffi)^{(k)}\|_1      + e^{k\theta}\| \widehat{r}\,\ffi  \|_1\right]\le\nonumber\\
&
\le|t-t_0|\left[ \suml_{k=0}^2 C_k(t_0)\mu_\theta(\cR^d)\|\ffi^{(k)}\|_1      + e^{k\theta}\| \widehat{r}\,\ffi  \|_1\right]=\nonumber\\
&
=|t-t_0| K(t_0,\ffi)
\end{align}
with some constants  $C_k(t_0)$  depending only on $t_0$ and $K(t_0,\ffi)$ depending only on $t_0$ and $\ffi$.  These constants $C_k(t_0)$ arise from the calculations with  the operator  $\left(\widehat{he^{-\theta h}}\right)_1(\cdot,D)$ obtained in the Lemma \ref{Lemma:H:str-cont+deriv-at-0}. Note, that all calculations in Lemma \ref{Lemma:H:str-cont+deriv-at-0} remain true for any $\ffi\in S(\cR^d)$. Moreover, by Assumption \ref{assum_smoothnessOfCoef} (iv) we have $r(\cdot)\in C^\infty(\cR^d)$ and, as a negative definite function, $r(\cdot)$ grows at infinity with all its derivatives not faster than a polynomial (cf. Lemma 3.6.22 and Theo.3.7.13 in \cite{NJ}). Therefore, $\widehat{r^m}\ffi$,  $\left(\widehat{r^me^{-\theta r}}\right)\ffi\in S(\cR^d)$ for any $\ffi\in C_c^\infty(\cR^d)\subset S(\cR^d)$ and any $m\in \mathbb{N}\cup\{0\}$. In the same way by Lemma \ref{lemma:order of PDO} and Lemma \ref{Lemma:H:str-cont+deriv-at-0} with $[0,1]\ni t\to 0$ and   $\theta\in(0,t)$ we obtain
\begin{align*}
&\left\|\frac{F_1(t)\ffi-\ffi}{t}+ \widehat{H}_1(\cdot,D)\ffi \right\|_1=t\left\|\left(\widehat{ H^2e^{-\theta H} } \right)_1(\cdot,D)\ffi \right\|_1\le\\
&
\le t\left\|\left(\widehat{ h^2e^{-\theta h} } \right)_1(\cdot,D)\circ \left(\widehat{e^{-\theta r}}\right)\ffi+2 \left(\widehat{ he^{-\theta h} } \right)_1(\cdot,D)\circ \left(\widehat{re^{-\theta r}}\right)\ffi+\left(\widehat{ e^{-\theta h} } \right)_1(\cdot,D)\circ \left(\widehat{r^2e^{-\theta r}}\right)\ffi\right\|_1=\\
&
=t\left\|\left(\widehat{ h^2e^{-\theta h} } \right)_1(\cdot,D)(\mu_\theta*\ffi)+2 \left(\widehat{ he^{-\theta h} } \right)_1(\cdot,D)(\mu_\theta*[\widehat{r}\ffi])+F_1(\theta)(\mu_\theta*[\widehat{r^2}\ffi])\right\|_1\le\\
&
\le  t\mu_\theta(\cR^d)\left[\suml_{k=0}^4 C_k \|\ffi^{(k)}\|_1+ 2\suml_{k=0}^2 C'_k \|(\widehat{r}\ffi)^{(k)}\|_1       +e^{k\theta}\left\|\widehat{(r^2)}\ffi\right\|_1\right]\le\\
&
=t K'(\ffi).
\end{align*}

\end{proof}

\begin{theorem}
Let $X=L_1(\cR^d)$. Let $\tau=1$.
 Under Assumptions \ref{assum_smoothnessOfCoef} (i), (iii), (iv) and Assumption  \ref{assum_generate} the family $(F_\tau(t))_{t\ge0}$ given by the formula~\eqref{F0} is Chernoff equivalent to the semigroup $(T^\tau_t)_{t\ge0}$, generated by the closure $(L^\tau,\D(L^\tau))$ of a $\psi$DO $(\widehat{H}_\tau(\cdot,D),C^\infty_c(\cR^d))$ with the $\tau$-symbol $H(q,p)$ as in \eqref{def:H}. Therefore, the  Feynman formula
 $$
 (T^\tau_t)\ffi=\lim_{n\to\infty}(F_\tau(t/n))^n\ffi
 $$   holds in $X=L_1(\cR^d)$ locally uniformly with respect to  $t$.
The  obtained Feynman formula converts also into a Lagrangian Feynman formula
\begin{align}\label{formula:LFF-1}
&(T^\tau_t)\ffi(q_0)=\\
&
=\lim_{n\to\infty}\intl_{\cR^{2nd}}\exp\left\{-\suml_{k=1}^n \frac{A^{-1}(q_{k-1})(q_{k}-q_{k-1}+z_k-b(q_{k-1})t/n)\cdot(q_{k}-q_{k-1}+z_k-b(q_{k-1})t/n)}{4t/n}\right\}\times\nonumber\\
&
\times\prodl_{k=1}^n \left((2\pi t/n)^d\det A(q_{k-1})\right)^{-1/2} \exp\left\{-\frac{t}{n}\suml_{k=1}^n c(q_{k-1}) \right\}\ffi(q_n)dq_1\,\mu_{t/n}(dz_1)\ldots dq_n\,\mu_{t/n}(dz_n)\nonumber
\end{align}
and with $\ffi\in S(\cR^d)$ into a Hamiltonian Feynman formula
\begin{align}\label{HFF-1}
&(T^\tau_t)\ffi(q_0)=\\
&
=\lim_{n\to\infty}\frac{1}{(2\pi)^d}\intl_{\cR^{2nd}}\exp\left\{i\suml_{k=1}^n p_k\cdot (q_{k+1}-q_k)\right\} \exp\left\{-\frac{t}{n}\suml_{k=1}^n H(q_{k+1},p_k)\right\}\ffi(q_1)dq_1dp_1\ldots dq_ndp_n,\nonumber
\end{align}
where  $q_{n+1}:=q_0$  in the pre-limit expressions for each $n\in \mathbb{N}$.
\end{theorem}
\begin{proof}
The statement of the Theorem is a straightforward consequence of Chernoff Theorem~\ref{Chernoff}, Lemma~\ref{Lemma_tau-norm estim} and Lemma~\ref{Lemma:H:str-cont+deriv-at-0}.  Lagrangian Feynman formula is obtained with the help of representation \eqref{formula:LagrangianFamily}. Under Assumptions \ref{assum_smoothnessOfCoef} (i), (iii), (iv) we have $F_\tau(t):S(\cR^d)\to S(\cR^d)$ (due to Lemma 3.3 in \cite{BSchS-IDAQPRT}). Therefore, all expressions in the right hand side of the Hamiltonian Feynman formula are well defined.
\end{proof}

\section{The construction of phase space Feynman path integrals via a family of Hamiltonian Feynman pseudomeasures for $\tau\in[0,1]$.}
A Feynman pseudomeasure on a (usually infinite dimensional) vector space is a  continuous linear functional on a locally convex space of some functions defined on this vector space.  The value of this functional on a function belonging to  its domain  is called Feynman integral with respect to this Feynman pseudomeasure. If the considered vector space  is itself a set of functions taking values in classical configuration or phase space then the corresponding Feynman integral is called configuration or phase space Feynman path integral.

There are many approaches for giving a mathematically rigorous meaning  to  phase space Feynman path integrals.  Some phase space Feynman path integrals are defined via the Fourier transform  and    via Parseval's equality (see \cite{Smo-Shavg},  cf. \cite{AlH-KMa08}; see \cite{SSham},  \cite{DWMMN79}, \cite{CDWM06} and references therein); some are defined via an analytic continuation of a Gaussian measure on the set of paths in a phase space \cite{Smo-Shavg},  some ---  via   regularization procedures, e.g., as  limits of integrals with respect to Gaussian measures with a diverging diffusion constant \cite{DauKl85};   the integrands of some phase space Feynman path  integrals are realized as Hida distributions in the setting of White Noise Ana\-lysis  \cite{BG11}.
     A variety of approaches  treats Feynman path integrals as  limits of  integrals over some  finite dimensional subspaces of paths when the dimension tends to infinity. Such path integrals are sometimes called sequential and are most convenient for direct calculations.  The general definition of a sequential Feynman pseudomeasure (Feynman path integral) in an abstract space (on a set of paths in a phase space, in particular)  can be found in \cite{Smo-Shavg}. Some concrete realizations  are e.g.\ presented in  \cite{STT}, \cite{BS},  \cite{AGM03}, \cite{Ich00}, \cite{KgFjw08}, \cite{Kg96}, \cite{KiKg81},   \cite{Gar66}.

One of the most convenient  definitions  of a Feynman pseudomeasure on a  conceptual level is via its Fourier transform.
Let $X$ be a locally convex space and $X^*$ be the set of all continuous linear functionals on $X$. Let $E$ be a real vector space and for all
$x\in E$ and any linear functional   $g$ on  $E$  let
$\phi_{g}(x)=e^{ig(x)}$. Let  $F_{E}$  be a locally convex set of some complex valued functions on   $E$. Elements of the set
  $F(E)^*$ are called \emph{ $F(E)^*$-distributions on  $E$} or just \emph{distributions on  $E$} (if we don't specify the space   $F(E)^*$ exactly).
Let   $G$  be a vector space of some linear functionals on   $E$ distinguishing elements of  $E$ and let
$\phi_{g}\in F_{E}$ for all   $g\in G$.  Then \emph{  $G$-Fourier transform} of an element
 $\eta\in F_{E}^{*}$ is a function  on   $G$  denoted by
 $\widetilde\eta$ or $\mathcal F[\eta]$ and defined by the formula $$\widetilde\eta(g)\equiv\mathcal F[\eta](g):=\eta(\phi_{g}).$$
If a set   $\{\phi_{g}:g\in G\}$ is total in $F_{E}$ (i.e.  its linear span is dense in  $F_{E}$) then any element    $\eta$ is uniquely defined by its Fourier transform.

\begin{definition}
 Let   $b$ be a quadratic functional on   $G$,
$a\in E$ and  $\alpha\in \mathbb C$. Then  Feynman $\alpha$-pseudomeasure on  $E$ with correlation functional  $b$ and mean  $a$  is a distribution
 $\Phi_{b,a,\alpha}$ on  $E$ whose Fourier transform is given by the formula   $$\mathcal F[\Phi_{b,a,\alpha}](g)=\exp\bigg\{\frac{\alpha
b(g)}{2}+ig(a)\bigg\}.$$
\end{definition}
If   $\alpha=-1$  and $b(x)\ge0$ for all $x\in G$ then  Feynman
$\alpha$-pseudomeasure is a Gaussian $G$-cylindrical measure on   $E$ (which however can be not $\sigma-$additive). If  $\alpha=i$ then we have a ``standard'' Feynman pseudomeasure which is usually used for solving Schr\"{o}dinger type equations.  In the sequel we will consider only these ``standard'' Feynman   $i-$pseudomeasures with   $a=0$.

\begin{definition}[Hamiltonian (or phase space) Feynman pseudomeasure]
Let   $E=Q\times P$, where  $Q$ and  $P$ are locally convex spaces,  $Q=P^{*}$, $P=Q^{*}$ (as vector spaces) with the duality $\langle\cdot,\cdot\rangle$;
the space   $G=P\times Q$ is identified with the space of all linear functionals on $E$ in the following way: for any  $g=(p_{g},q_{g})\in G$ and $x=(q,p)\in
E$ let $g(x)=\langle q,p_g\rangle +\langle q_g,p\rangle$.  Then   \textbf{Hamiltonian (or symplectic, or phase space) Feynman pseudomeasure} on   $E$  is a Feynman $i$-pseudomeasure $\Phi$ on   $E$ whose correlation functional   $b$ is given by the formula  $b(p_{g},q_{g})=2\langle q_g, p_g\rangle$ and mean $a=0$, i.e. $$\mathcal F[\Phi](g)=\exp\big\{i\langle q_g, p_g\rangle)\big\}.$$
\end{definition}

\begin{definition}
Assume that there exists a linear injective mapping  $B: G\to E$ such that   $b(g)=g(B(g))$ for all  $g\in G$ ($B$ is called \emph{correlation operator} of Feynman pseudomeasure). Let   $\D(B^{-1})$ be the domain of   $B^{-1}$. A function $\D(B^{-1})\ni x\mapsto
f(x)=e^{\frac{\alpha^{-1}B^{-1}(x)(x)}{2}}$ is called the \textbf{generalized density} of Feynman $\alpha-$pseudomeasure  (cf. \cite{SWW1}).
\end{definition}

\begin{example}
\noindent(i) If  $E=\cR^d=G$  then  the Feynman $i$-pseudomeasure   on  $E$ with correlation operator
 $B$ can be identified with a complex-valued measure (with unbounded variation) on a
$\delta$-ring of bounded Borel subsets of    $\cR^d$ whose density with respect to the Lebesgue measure is
 $f(x)=e^{-\frac{i}{2}(B^{-1}x,x)}$. In this case the generalized density coincides with the density in usual sense.

\noindent (ii) If we consider the Hamiltonian Feynman pseudomeasure on $E=Q\times P$ then take $B: (p,q)\in G\subset E^*\to (q,p)\in E$.  Then  we have $g(B(g))=g(B(p_g,q_g))=g(q_g,p_g)=2\langle q_g,p_g\rangle=b(g)$. Moreover, $B^{-1}:E\to E^*$ is defined by the formula $B^{-1}(q,p)=(p,q)$ and hence the generalized density of the Hamiltonian Feynman pseudomeasure is given by the formula $f(q,p)=\exp\{i\langle q,p\rangle\}$.
\end{example}

The concepts  given above allow to introduce the following definition of a Feynman pseudomeasure in the frame of sequential approach (in the sequel we assume  any standard regularization of oscillating  integrals, e.g.,
 $\intl_{E}f(z)dz=\liml_{\varepsilon\to0}\intl_{E}f(z)e^{-\varepsilon|z|^2}dz$).

\begin{definition}[Sequential Feynman pseudomeasure]\label{Def-sequential} Let  $\{E_n\}_{ n\in
\mathbb{N}}$ be an increasing sequence of finite dimensional subsets of   $\D(B^{-1})$. Then the value of a sequential Feynman
$\alpha-$pseudomeasure  $\Phi^{\{E_n\}}_{B,\alpha}$ (with mean $a=0$) associated with the sequence
 $\{E_n\}_{ n\in \mathbb{N}}$ on a function  $\psi:E\to\mathbb{C}$ (this value is called sequential  Feynman integral of  $\psi$) is defined by the formula
$$\Phi^{\{E_n\}}_{B,\alpha}(\psi)= \liml_{n\to\infty}\bigg(\intl_{E_n}e^{\frac{\alpha^{-1}B^{-1}(x)(x)}{2}}dx\bigg)^{-1}
\intl_{E_n}\psi(x)e^{\frac{\alpha^{-1}B^{-1}(x)(x)}{2}}dx,$$ where one integrates with respect to the Lebesgue measure on $E_n$, if the limit in the r.h.s. exists.
\end{definition}
The fact that a function $\psi$ belongs to the domain of the functional $\Phi^{\{E_n\}}$ depends only on restrictions of this function to the subspaces $E_n$.
In the particular case of Hamiltonian Feynman pseudomeasure Definition \ref{Def-sequential} can be read as follows:

\begin{definition}\label{HFPI1}
 Let $\{E_n=Q_n\times P_n\}_{n\in \mathds{N}}$ be an
increasing sequence of finite dimensional vector subspaces of
$E=Q\times P$, where $Q_n$ and $P_n$ are vector subspaces of $Q$ and
$P$ respectively. The value $\Phi_{\{E_n\}}(G)$ of the Hamiltonian Feynman
pseudomeasure  $\Phi_{\{E_n\}}$, associated with the sequence
$\{E_n\}_{n\in \mathds{N}}$, on a function $\psi:E\to\comp$, i.e.\  a Feynman path integral of $\psi$, is defined by the formula
\begin{equation}
    \Phi_{\{E_n\}}(\psi)
    =\liml_{n\to\infty}\Big( \intl_{E_n}e^{i\langle p,q \rangle}\,dq\,dp \Big)^{-1}\intl_{E_n}\psi(q,p)e^{i\langle p, q\rangle}\,dq\,dp,
\end{equation}
if this limit exists. In this formula (as well as before) all integrals must be considered in a suitably regularized sense.
\end{definition}


\bigskip

In the sequel we present a construction of the Hamiltonian
Feynman pseudomeasure for a particular family of spaces $E^{x,\tau}_t$ with $\tau\in[0,1]$, cf.\ \cite{STT}, \cite{BBSchS}. For any $t>0$ let $PC([0,t],
\real^d)$ be the vector space of all functions on $[0,t]$ taking
values in $\real^d$ whose distributional derivatives are measures with
finite support. Let $PC^l([0,t], \real^d)$ denote the space of all
left continuous functions from $PC([0,t], \real^d)$. 
 Let $PC^\tau([0,t], \real^d)$ be the collection of functions $f$ from $PC([0,t],\cR^d)$ such that  for all $s\in(0,t)$
\begin{equation}\label{tau-continuity}
f(s)=\tau f(s+0)+(1-\tau)f(s-0).
\end{equation}

For each $x\in\real^d$ let
$$
{Q}_t^{x,\tau}=\{ f\in
PC^\tau([0,t], \real^d):\,\, f(0)=\liml_{t\to+0}f(t),\,\, f(t)=x\},
$$
$${P}_t=\{ f\in PC^l([0,t], \real^d):\, f(0)=\liml_{t\to+0} f(t) \}
$$
 and
$E_t^{x,\tau}={Q}^{x,\tau}_t\times {P}_t$. The spaces ${Q}^{x,\tau}_t$ and ${P}_t$ are
taken in duality by the form: $$\langle
q(\cdot),p(\cdot)\rangle\mapsto\intl_0^t p(s)\dot{q}(s)\,ds,$$ where
$\dot{q}(s)\,ds$ denotes the measure which is the distributional derivative
of $q(\cdot)$. We will consider the elements of $E_t^{x,\tau} $ as
functions taking values in ${\mathbf{E}}=\mathbf{Q}\times
\mathbf{P}=\real^d\times\real^d$.

Let $t_0=0$  and for any $n\in \mathds{N}$  and any $k\in
\mathds{N}$, $k\le n$,  let $t_k=\frac{k}{n}t$.
 Let $F_n\subset PC([0,t],\real^d)$  be the
space of functions, the  restrictions of which to any interval
$(t_{k-1},t_k)$ are constant functions.
Let  $Q^\tau_n=F_n\cap
Q^{x,\tau}_t$, $P_n=F_n\cap P_t$. Let $J^\tau_n$ be
the mapping of $E^\tau_n=Q^\tau_n\times P_n$ to $(\real^d\times\real^d)^n$,
defined by
\begin{align*}
    J^\tau_n(q,p)&=\big(q(\tfrac{t}{n}-0),p(\tfrac{t}{n}),\ldots,q(\tfrac{(n-1)t}{n}-0),p(\tfrac{(n-1)t}{n}),q(\tfrac{nt}{n}-0),p(\tfrac{nt}{n} )\big)\equiv\\
    &\equiv\big(q_1,p_1,.....,q_n,p_n\big).
\end{align*}

The map $J^\tau_n$ is a one-to-one correspondence of $E^\tau_n$ and
$(\real^d\times\real^d)^n$. Therefore,  in this particular case Definition \ref{HFPI1} can be rewritten in the following way:
\begin{definition}\label{DefHFPI} The Hamiltonian (or phase space) Feynman path integral
$$
    \Phi^\tau_x(\psi)\equiv\intl_{E^{x,\tau}_t}\psi(q,p)\Phi^\tau_x(dq,dp)
    \equiv
    \intl_{E^{x,\tau}_t} \psi(q(s),p(s))e^{i\intl_0^t p(s) \dot{q}(s)ds}\prodl_{\tau=0}^t\,dq(s)\,dp(s)
$$
of a function $\psi:Q^{x,\tau}_t\times P_t\to\real$ is defined as a limit:
\begin{equation}\begin{aligned}\label{formula:FPIapprox}
    \Phi^\tau_x(\psi)
    &=\liml_{n\to\infty}\frac{1}{(2\pi)^{dn}}\intl_{(\real^d\times\real^d)^n} \psi((J^\tau_n)^{-1}(q_1,p_1,.....,q_n,p_n)) \times\\ &\qquad\qquad\times\exp\left[i\suml_{k=1}^n p_k\cdot(q_{k+1}-q_{k})\right]
    dq_1\,dp_1.......dq_n\,dp_n,
\end{aligned}
\end{equation}
where $q_{n+1}:=x$ in each pre-limit expression.
\end{definition}
\begin{remark}
The  generalized
density of the pseudomeasure $\Phi^\tau_x$ can be defined through the formula
$$
    \intl_{E^{x,\tau}_t}\psi(q(s),p(s))\Phi^\tau_x(dq\,dp)
    =\liml_{n\to\infty} C_n \intl_{Q^\tau_n\times P_n} \psi(q(s),p(s))\,\exp\left[i\intl_0^t p(s) \dot{q}(s)\,ds\right]\nu_n(dq)\,\nu_n(dp),
$$
where  $(C_n)^{-1}=\int_{Q^\tau_n\times P_n} \exp\left[i\int_0^t p(s)\dot{q}(s)\,ds\right]\nu_n(dq)\,\nu_n(dp)$ and $\nu_n$ is Lebesgue measure.
\end{remark}
\begin{remark}
The construction described above is a modification and an extension  of  constructions introduced  in \cite{STT},   \cite{OST}, \cite{BS} and \cite{BBSchS}.
\end{remark}

\section{Phase space  Feynman path integrals for evolution semigroups generated by $\tau$-quatization of some L\'{e}vy-Khintchine type Hamilton functions for $\tau\in[0,1]$}

In this section we show that Hamiltonian Feynman formulae obtained above can be interpreted as some phase space Feynman path integrals. Therefore, the corresponding phase space Feynman path integrals do exist and coincide  with some functional integrals with respect to countably additive (mainly probability) measures  associated with some  Feller type semigroups.

\begin{theorem}\label{Theo:FPI-tau} Let  $\tau\in[0,1]$ and $H(q,p)=h(q,p)$, where $h(q,p)$ is given by the formula \eqref{def:h}.
 Let  Assumption \ref{assum_smoothnessOfCoef} (i), (iii) and Assumption \ref{assum_generate} fulfil. Let $(T^\tau_t)_{t\ge0}$ be the semigroup generated by the closure $(L^\tau,\D(L^\tau))$ of a $\psi$DO $(\widehat{H}_\tau(\cdot,D),C^\infty_c(\cR^d))$ with the $\tau$-symbol $H(q,p)$. Then
the Hamiltonian Feynman formula  \eqref{hff-tau} can be interpreted as  a phase space Feynman path integral
\begin{align}\label{formula:FPI-tau}
T^\tau_t\ffi(x)=\intl_{E^{x,\tau}_t}e^{-\intl_0^t H(q(s),p(s))ds}\ffi(q(0))\Phi^\tau_x(dqdp).
\end{align}
\end{theorem}
\begin{proof} Indeed, using Definition \ref{DefHFPI} we get
\begin{equation*}
\begin{aligned}
    (T^\tau_t\varphi)(x)
    &
    =\lim\limits_{n\rightarrow\infty}(2\pi)^{-dn}\int\limits_{(\mathds{R}^d)^{2n}} \exp\left(i\sum\limits_{k=1}^n p_k\cdot(q_{k+1}-q_{k})\right)\times\\
    &
    \qquad\times \exp\left(-\frac{t}{n}\sum\limits_{k=1}^n H(\tau q_{k+1}+(1-\tau)q_k,p_k)\right)\varphi(q_1)\,dq_1\,dp_1\cdots dq_n\,dp_n=\\
    &
    =\intl_{E^{x,\tau}_t}e^{-\intl_0^t H(q(s),p(s))ds}\ffi(q(0))\Phi^\tau_x(dqdp),
\end{aligned}
\end{equation*}
where,  in each pre-limit expression in the Hamilton formula, we have $q_{n+1}:=x$; moreover, $q_1=q(t/n-0)\to q(0)$ as  $n\to\infty$ due to the definition of the space  ${Q}_t^{x,\tau}$ and
 $$\frac{t}{n}\sum\limits_{k=1}^nH(\tau q_{k+1}+(1-\tau)q_k,p_k)=\sum\limits_{k=1}^n H( q(t_k),p(t_k))(t_k-t_{k-1})\to \intl_0^t H(q(s),p(s))ds$$
since  any path $(q(s),p(s))\in E^{x,\tau}_t$ is piecewise continuous and has a finite number of jumps, $H$ is a continuous function.
\end{proof}
\begin{remark}
Note, that the integrand in the Feynman path integral \eqref{formula:FPI-tau} is the same for all $\tau\in[0,1]$, only the space $E^{x,\tau}_t$, defining the sequential pseudomeasure $\Phi^\tau_x$ is different; this space contains those paths $q(s)$ which are ``$\tau$-continuous'' (see the formula \eqref{tau-continuity} for the definition).
\end{remark}
\begin{remark} Under Assumptions  \ref{assum_smoothnessOfCoef} (i), (ii), Assumption \ref{assum_generate} and due to the formula \eqref{formula:tau-to-1}  (i.e. Lemma 2.1 in \cite{BBSchS})  we see that $\widehat{H}_\tau(q,D)\ffi(q)=\widehat{H}^\tau_1(q,D)\ffi(q)$, where $\widehat{H}^\tau_1(q,D)\ffi(q)$ is
 a pseudo-differential operator  with 1-symbol
$$H^\tau(q,p)=c_\tau(q)+ ib_\tau(q)\cdot p  +p\cdot A(q)p $$ and we have
 $$b_\tau(q)=b(q)-2(1-\tau)\Div A(q),$$
 $$c_\tau(q)=c(q)+(1-\tau)\Div b(q)-(1-\tau)^2\tr(\Hess A(q)).$$
Therefore, due to Theorem \ref{Theo:FF-tau}   and Theorem \ref{Theo:FPI-tau}  there is a kind of ``change of variable formula'' for $H(q,p)=c(q)+ib(q)\cdot p+p\cdot A(q)p $:
\begin{align*}
T^\tau_t\varphi(x) &=\intl_{E^{x,1}_t}\exp\left[-\intl_0^t H^\tau(q(s),p(s)) ds\right]\varphi(q(0))\,\Phi^1_x(dq\,dp)=\\
    &=\intl_{E^{x,\tau}_t}\exp\bigg[-\intl_0^t H(q(s),p(s))ds\bigg]\ffi(q(0))\Phi^\tau_x(dqdp).
\end{align*}
i.e.,
\begin{align*}
&T^\tau_t\varphi(x)=\\
    &=\intl_{E^x_t}\exp\left[-\intl_0^t p(s)\cdot A(q(s))p(s)\,ds\right] \exp\left[-i\intl_0^t\Big[(b(q(s))-2(1-\tau)\Div A(q(s)))\cdot p(s) \Big] ds\right]\times\\
    &\quad\times \exp\left[-\intl_0^t \Big[c(q(s))+(1-\tau)\Div b(q(s))-(1-\tau)^2\tr(\Hess A(q(s)))\Big]ds\right] \varphi(q(0))\Phi^1_x(dq\,dp)=\\
    &=\intl_{E^{x,\tau}_t}\exp\bigg[-\intl_0^t p(s)\cdot A(q(s))p(s)\,ds-i\intl_0^t b(q(s))\cdot p(s) \,ds-\intl_0^t c(q(s))ds\bigg]\ffi(q(0))\Phi^\tau_x(dqdp).
\end{align*}

\end{remark}

Due  to Definition \ref{DefHFPI} the
Hamiltonian Feynman formula \eqref{HFF-1}  can be
interpreted as a Hamiltonian Feynman path integral with respect to
the Feynman pseudomeasure $\Phi^1_x$. Therefore the following theorem is true (cf. \cite{BBSchS}).
\begin{theorem}
\label{Th-FPI}
Let $\tau=1$.  Under Assumptions \ref{assum_smoothnessOfCoef} (i), (iii), (iv) and Assumption  \ref{assum_generate} the  semigroup $(T^\tau_t)_{t\ge0}$, generated by the closure $(L^\tau,\D(L^\tau))$ of a $\psi$DO $(\widehat{H}_\tau(\cdot,D),C^\infty_c(\cR^d))$ with the $\tau$-symbol $H(q,p)$ as in \eqref{def:H}
  can be represented by a Hamiltonian Feynman path integral with respect to
the Feynman pseudomeasure $\Phi^1_x$:
\begin{equation}\label{FPI-1}
    T_t\varphi(x)
    =\intl_{E^{x,1}_t}e^{-\intl_0^t H(q(s),p(s)) ds}\varphi(q(0))\,\Phi^1_x(dq\,dp).
\end{equation}
\end{theorem}

\begin{remark}
Note, that our definition of the space $E^{x,1}_t$ differs from the definition of the space $E^{x}_t$  in the paper \cite{BBSchS}.  Hence the corresponding sequential Feynman pseudomeasures used above and in \cite{BBSchS} are also different. This leads to different Feynman path integrals representing the considered evolution semigroup (cf. with Theorem 3.5 in \cite{BBSchS}).
\end{remark}

\begin{remark}
Lagrangian Feynman formula \eqref{formula:LFF-tau} (resp. \eqref{formula:LFF-1}) actually  provides a tool to compute Feynman path integral \eqref{formula:FPI-tau} (resp. \eqref{FPI-1}). The limits in both Lagrangian formulas coincide with  functional integrals over some probability measures.
\end{remark}

\begin{ack}
Financial support by the DFG through the project GR-1809/10-1, by  the Ministry of education and science of Russian Federation  through the project 14.B37.21.0370,   by the President of Russian Federation through the project MK-4255.2012.1 and by DAAD is gratefully acknowledged.
\end{ack}

\end{document}